\documentclass[reqno,11pt]{amsart}
\usepackage{amssymb,amsmath,mathrsfs}
\usepackage{graphics,tikz,enumerate}

\numberwithin{equation}{section}
\newtheorem{theorem}{Theorem}
\newtheorem{lemma}[theorem]{Lemma}

\newtheorem{corollary}[theorem]{Corollary}

\theoremstyle{definition}

\newtheorem*{remark}{Remark}

\usetikzlibrary{patterns}
\usetikzlibrary{decorations.pathreplacing}
\definecolor{tileColor}{HTML}{A32400}
\definecolor{bandColor}{HTML}{A32400}
\definecolor{spotColor}{HTML}{A32400}
\definecolor{spanColor}{HTML}{0024A3}
\definecolor{arcColor}{HTML}{A32400}

\tikzstyle{dottedBand}=[draw, line width=1.5, dotted, color=bandColor]
\tikzstyle{solidBand}=[draw, line width=1.5, color=bandColor]
\tikzstyle{Tiling}=[line width=1.5]
\tikzstyle{Remover}=[line width=1.6, color=white]
\tikzstyle{dottedSep}=[draw, line width=1.5, dotted, color=tileColor]
\tikzstyle{solidSep}=[draw, line width=1.5, color=tileColor]

\tikzstyle{Tree}=[line width=1.5, color=spanColor]
\tikzstyle{spanRem}=[line width=1.65, color=white]
\tikzstyle{node}=[circle, inner sep=1.5, fill=spanColor]
\tikzstyle{arc1}=[draw, line width=1.2, color=arcColor]

\newcommand{\plainTile}[2]{
\def\n{#1}
\def\sf{#2}
\begin{tikzpicture}[scale=\sf]
\draw[lightgray] (0,0) grid (\n,1);
\draw[Tiling] (0,0) rectangle (\n,1);
\end{tikzpicture}
}

\newcommand{\dottedTile}[3]{
\def\n{#1}
\def\sf{#3}
\begin{tikzpicture}[scale=\sf]
\draw[lightgray] (0,0) grid (\n,1);
\foreach \x in {#2}{
	\draw[dottedSep] (\x,0.04) -- (\x,1);
}
\draw[Tiling] (0,0) rectangle (\n,1);
\end{tikzpicture}
}

\newcommand{\mixedTile}[4]{
\def\n{#1}
\def\sf{#4}
\begin{tikzpicture}[scale=\sf]
\draw[lightgray] (0,0) grid (\n,1);
\foreach \x in {#2}{
	\draw[Tiling] (\x,0) -- (\x,1);
}
\foreach \x in {#3}{
	\draw[dottedSep] (\x,0.025) -- (\x,1);
}
\draw[Tiling] (0,0) rectangle (\n,1);
\end{tikzpicture}
}

\newcommand{\spottedTile}[5]{
\def\n{#1}
\def\sf{#5}
\begin{tikzpicture}[scale=\sf]
\draw[lightgray] (0,0) grid (\n,1);
\foreach \x in {#2}{
	\draw[Tiling] (\x,0) -- (\x,1);
}
\foreach \x in {#3}{
	\draw[dottedSep] (\x,0.06) -- (\x,1);
}
\foreach \x in {#4}{
	\draw[color=spotColor,fill=spotColor] (\x-0.5,0.5) circle (3pt); 
}
\draw[Tiling] (0,0) rectangle (\n,1);
\end{tikzpicture}
}

\newcommand{\arcTile}[5]{
\def\n{#1}
\def\sf{#5}
\begin{tikzpicture}[scale=\sf]
\draw[lightgray] (0,0) grid (\n,1);
\foreach \x in {#2}{
	\draw[Tiling] (\x,0) -- (\x,1);
}
\foreach \x in {#3}{
	\draw[dottedSep] (\x,0.06) -- (\x,1);
}
\foreach \i in {1,...,\n} {%
	\draw[fill,arc1] (\i-0.5,0.5) circle (0.06);
}
\foreach \x/\y in {#4} {
	\pgfmathsetmacro\s{\x+\y}
	\pgfmathsetmacro\d{\y-\x}
	\pgfmathsetmacro\eps{\d^0.6}
	\draw[thick,arc1] (\x-0.5,0.5) parabola bend (0.5*\s-0.5,0.5+0.18*\eps) (\y-0.5,0.5);
}
\draw[Tiling] (0,0) rectangle (\n,1);
\end{tikzpicture}
}

\newcommand{\spanTree}[5]{ 
\def\n{#1}
\def\sf{#5}
\begin{tikzpicture}[scale=\sf]
\draw[Tree] (0.5,1) -- (\n+0.5,1);
\draw[Tree] (0.5,0) -- (\n+0.5,0);
\foreach \x in {#2}{
	\draw[spanRem] (\x-0.5,0) -- (\x+0.5,0);
}
\foreach \x in {#3}{
	\draw[spanRem] (\x-0.5,1) -- (\x+0.5,1);
}
\foreach \x in {#4}{
	\draw[Tree] (\x-0.5,0) -- (\x-0.5,1); 
}
\foreach \x in {0,...,\n}{ 
  \node[node] at (\x+0.5,0) {}; 
  \node[node] at (\x+0.5,1) {}; 
 }
\end{tikzpicture}
}

\newcommand{\drawarc}[3]{%
\def\sf{#1}
\def\n{#2}
\tikz[scale=\sf]{%
\foreach \i in {1,...,\n} {%
	\draw[fill,arc1] (\i,0) circle (0.06);
}
\foreach \x/\y in {#3} {
	\pgfmathsetmacro\s{\x+\y}
	\pgfmathsetmacro\d{\y-\x}
	\pgfmathsetmacro\eps{\d^0.7}
	\draw[thick,arc1] (\x,0) parabola bend (0.5*\s,0.25*\eps) (\y,0);
}
}}

\newcommand{\oeis}[1]{\cite[#1]{Sloane}}
\def\C{\mathscr{C}}

\def\P{\mathcal{P}}
\def\Q{\mathcal{Q}}
\def\U{\mathcal{U}}

\def\NCN{\textup{NCN}}
\def\NCNi{\textup{NCN}^{i}}

\title{Fibonacci colored compositions and applications}
\author{Juan B. Gil}
\author{Jessica A. Tomasko}
\address{Penn State Altoona\\ 3000 Ivyside Park\\ Altoona, PA 16601}

\begin{document}

\begin{abstract}
We study compositions whose parts are colored by subsequences of the Fibonacci numbers. We give explicit bijections between Fibonacci colored compositions and several combinatorial objects, including certain restricted ternary and quaternary words, spanning trees in the ladder graph, unimodal sequences covering an initial interval, and ordered-consecutive partition sequences. Our approach relies on the basic idea of representing the colored compositions as tilings of an $n$-board whose tiles are connected, and sometimes decorated, according to a suitable combinatorial interpretation of the given coloring sequence.
\end{abstract}

\maketitle

\section{Introduction}

A composition of a positive integer $n$ is an ordered tuple of positive integers, called parts, whose sum is $n$. There are $\binom{n-1}{m-1}$ compositions of $n$ with $m$ parts, and $2^{n-1}$ total compositions of $n$. More generally, given a sequence $w=(w_k)_{k\in\mathbb{N}}$ of nonnegative integers, a $w$-color composition of $n$ is a composition of $n$ such that part $k$ can take on $w_k$ colors. If $W_n$ is the number of $w$-color compositions of $n$, then the generating functions $w(t)=\sum_{n=1}^\infty w_n t^n$ and $W(t)=\sum_{n=1}^\infty W_n t^n$ are known to satisfy the relation $W(t) = \frac{w(t)}{1-w(t)}$. This means that $(W_n)_{n\in\mathbb{N}}$ is the {\sc invert} transform of $(w_n)_{n\in\mathbb{N}}$. Colored compositions appear naturally in a wide variety of combinatorial problems and are therefore of special interest (see e.g\ \cite{ABCM14, BGW18, GL13, HL68, BHop12,  MW61}). For a comprehensive study of compositions, we refer to the book by S.~Heubach and T.~Mansour \cite{HM10}.

In this paper, we will focus on the Fibonacci numbers \oeis{A000045} defined by
\begin{gather*}
 F_0=0,\;\; F_1=1,\\
 F_n = F_{n-1}+F_{n-2} \;\text{ for } n\ge 2,
\end{gather*}
and we will use various subsequences of $(F_k)_{k\in\mathbb{N}}$ as coloring sequences. Our strategy is to suitably represent the resulting colored compositions of $n$ as tilings of an $n$-board ($1 \times n$ rectangular board), where each part comes as a tile decorated by means of extra tiling patterns that we call {\em secondary tilings}. These tilings not only offer a deeper understanding of the colored objects, but they oftentimes lead to elegant bijective maps to other combinatorial families.

As a start, in Section~\ref{sec:basics}, we review some common bijections between three families of compositions counted by the Fibonacci numbers. The sets and techniques presented in these bijections will be implemented in later proofs.

In Section~\ref{sec:FiboColoring}, we consider Fibonacci colored compositions and introduce the idea of a secondary tiling, where we represent the coloring sequence by adopting the tilings introduced in Section~\ref{sec:basics}. The use of secondary tilings for representing Fibonacci colored compositions allows us to easily connect them to ternary words and other restricted compositions.

After discussing Fibonacci colored compositions, in Section~\ref{sec:evenFibo}, we consider compositions colored by the even-indexed Fibonacci numbers \oeis{A001906}. We use the spotted tilings introduced by Hopkins \cite{BHop12} to bijectively connect our colored compositions with other combinatorial structures, specifically $01$-avoiding words over the alphabet $\{0,1,2,3\}$, spanning trees in the ladder graph, and simple rectangular mazes. 

In addition, in Section~\ref{sec:oddFibo}, we discuss compositions colored by the odd-indexed Fibonacci numbers \oeis{A001519} and we connect them with the set of unimodal sequences covering an initial interval of positive integers as well as the set of ordered-consecutive partition sequences. To this end, we introduce a set of partitions (called {\em totally nested partitions}) equinumerous with the set of indecomposable partitions that are noncrossing and nonnesting. These partitions are then used to conveniently decorate the secondary tilings needed for the representation of the odd-indexed Fibonacci colored compositions.

Finally, in Section~\ref{sec:furtherRemarks}, we give a simple bijection between the Fibonacci colored compositions from Section~\ref{sec:FiboColoring} and a set of multicompositions recently studied by Hopkins and Ouvry \cite{HO20}.

\section{Compositions counted by the Fibonacci numbers}
\label{sec:basics}

Consider the following sets of integer compositions:
\begin{align*}
\C_{1,2}(n) &= \{\text{compositions of $n$ having only parts $1$ and $2$}\}, \\
\C_{\rm odd}(n) &= \{\text{compositions of $n$ having only odd parts}\}, \\
\C_{>1}(n) &= \{\text{compositions of $n$ having no part 1}\}.
\end{align*}

It is well known and easy to prove that
\begin{equation*}
   \left|\C_{1,2}(n)\right| = F_{n+1},\quad \left|\C_{\rm odd}(n)\right| = F_{n}, \quad \left|\C_{>1}(n)\right| = F_{n-1}.
\end{equation*}
As a warm-up exercise, we give two simple combinatorial bijections 
\[ \alpha: \C_{1,2}(n) \to \C_{\rm odd}(n+1) \;\text{ and }\; \beta: \C_{\rm odd}(n) \to \C_{>1}(n+1) \]
that rely on the tiling interpretation of compositions.

\medskip
\noindent\textbf{Bijection between $\C_{1,2}(n)$ and $\C_{\rm odd}(n+1)$.} \
Define $\alpha$ by the following algorithm:
\begin{enumerate}[(i)]
\item Given a composition $(j_1,\dots,j_k)\in \C_{1,2}(n)$, represent it as a tiling of an $n$-board;
\item draw a horizontal line through the middle of every domino (tile of length 2); \label{bundling}
\item draw a vertical line at the center of each domino and remove the top and bottom horizontal lines of the board;
\item we let $\alpha(j_1,\dots,j_k)$ be the composition of $n+1$ obtained by considering the vertical lines as parts after bundling those that are connected by the horizontal line segments drawn in step \eqref{bundling}. 
\end{enumerate}
The resulting composition contains only odd parts, hence $\alpha(j_1,\dots,j_k)\in \C_{\rm odd}(n+1)$.

\medskip
For example, given the composition $(1,1,2,1,2,2,1)$ in $\C_{1,2}(10)$, we get

\medskip
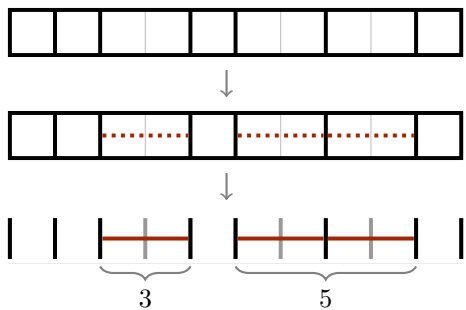
\begin{center}
\begin{tikzpicture}[scale=0.6]
\def\delta{0.05}
\begin{scope}[yshift=130]
\draw[lightgray] (0,0) grid (10,1);
\draw[Tiling] (0,0) rectangle (10,1);
\foreach \x in {1,2,4,5,7,9}{
\draw[Tiling] (\x,0) -- (\x,1);
}
\end{scope}

\begin{scope}[xshift=.3\textwidth,yshift=75]
\draw[thick,->,gray] (0,1.6) -- (0,1);
\end{scope}

\begin{scope}[yshift=65]
\draw[lightgray] (0,0) grid (10,1);
\draw[Tiling] (0,0) rectangle (10,1);
\foreach \x in {1,2,4,5,7,9}{
\draw[Tiling] (\x,0) -- (\x,1);
}
\foreach \x in {2,5,7}{
\draw[dottedBand] (\x+\delta,0.5) -- (\x+2-\delta,0.5);
}
\end{scope}

\begin{scope}[xshift=.3\textwidth,yshift=10]
\draw[thick,->,gray] (0,1.6) -- (0,1);
\end{scope}

\begin{scope}
\draw[Tiling,gray!80] (0,0) grid (10,1);
\draw[Tiling] (0,0) rectangle (10,1);
\foreach \x in {1,2,4,5,7,9}{
\draw[Tiling] (\x,0) -- (\x,1);
}
\foreach \x in {2,5,7}{
\draw[solidBand] (\x+\delta,0.5) -- (\x+2-\delta,0.5);
}
\draw[Remover] (-0.1,0) -- (10.1,0);
\draw[Remover] (-0.1,1) -- (10.1,1);
\draw [decorate,decoration={brace,amplitude=5pt,mirror,raise=2pt},thick,gray]
 (2,0) -- (4,0) node [midway,yshift=-14pt,black] {\small $3$};
\draw [decorate,decoration={brace,amplitude=5pt,mirror,raise=2pt},thick,gray]
 (5,0) -- (9,0) node [midway,yshift=-14pt,black] {\small $5$};
\end{scope}
\end{tikzpicture}
\end{center}
and we arrive at the composition $(1,1,3,5,1)$ in $\C_{\rm odd}(11)$. This algorithm is reversible.

\medskip
\noindent\textbf{Bijection between $\C_{\rm odd}(n)$ and $\C_{>1}(n+1)$.} \
Define $\beta$ by the following algorithm:
\begin{enumerate}[(i)]
\item Given a composition $(j_1,\dots,j_k)\in \C_{\rm odd}(n)$, represent it as a tiling of an $n$-board, 
and draw a horizontal line through the middle of each square tile; \label{onebundling}
\item within each of the tiles of length greater than 1, draw gray vertical lines to subdivide them into square tiles and draw horizontal lines through the middle of the subdivided squares in odd positions (from left to right); \label{oddbundling}
\item remove the top and bottom horizontal lines of the board;
\item we let $\beta(j_1,\dots,j_k)$ be the composition of $n+1$ obtained by considering the vertical lines as parts after bundling those that are connected by the horizontal line segments drawn in steps \eqref{onebundling} and \eqref{oddbundling}. 
\end{enumerate}

\medskip
For example, given the composition $(1,1,3,5,1)$ in $\C_{\rm odd}(11)$, we get

\medskip
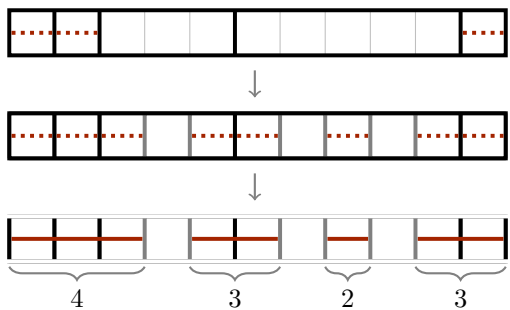
\begin{center}
\begin{tikzpicture}[scale=0.6]
\def\delta{0.05}
\begin{scope}[yshift=130]
\draw[lightgray] (0,0) grid (11,1);
\draw[Tiling] (0,0) rectangle (11,1);
\foreach \x in {1,2,5,10}{
\draw[Tiling] (\x,0) -- (\x,1);
}
\foreach \x in {0,1,10}{
\draw[dottedBand] (\x+\delta,0.5) -- (\x+1-\delta,0.5);
}
\end{scope}

\begin{scope}[xshift=.34\textwidth,yshift=75]
\draw[thick,->,gray] (0,1.6) -- (0,1);
\end{scope}

\begin{scope}[yshift=65]
\draw[Tiling, gray] (0,0) grid (11,1);
\draw[Tiling] (0,0) rectangle (11,1);
\foreach \x in {1,2,5,10}{
\draw[Tiling] (\x,0) -- (\x,1);
}
\foreach \x in {0,1,2,4,5,7,9,10}{
\draw[dottedBand] (\x+\delta,0.5) -- (\x+1-\delta,0.5);
}
\end{scope}

\begin{scope}[xshift=.34\textwidth,yshift=10]
\draw[thick,->,gray] (0,1.6) -- (0,1);
\end{scope}

\begin{scope}
\draw[Tiling, gray] (0,0) grid (11,1);
\draw[Tiling] (0,0) rectangle (11,1);
\foreach \x in {1,2,5,10}{
\draw[Tiling] (\x,0) -- (\x,1);
}
\draw[solidBand] (0+\delta,0.5) -- (3-\delta,0.5);
\draw[solidBand] (4+\delta,0.5) -- (6-\delta,0.5);
\draw[solidBand] (7+\delta,0.5) -- (8-\delta,0.5);
\draw[solidBand] (9+\delta,0.5) -- (11-\delta,0.5);

\draw[Remover] (-0.1,0) -- (11.1,0);
\draw[Remover] (-0.1,1) -- (11.1,1);
\draw [decorate,decoration={brace,amplitude=5pt,mirror,raise=2pt},thick,gray]
 (0,0) -- (3,0) node [midway,yshift=-14pt,black] {\small $4$};
\draw [decorate,decoration={brace,amplitude=5pt,mirror,raise=2pt},thick,gray]
 (4,0) -- (6,0) node [midway,yshift=-14pt,black] {\small $3$};
\draw [decorate,decoration={brace,amplitude=5pt,mirror,raise=2pt},thick,gray]
 (7,0) -- (8,0) node [midway,yshift=-14pt,black] {\small $2$};
\draw [decorate,decoration={brace,amplitude=5pt,mirror,raise=2pt},thick,gray]
 (9,0) -- (11,0) node [midway,yshift=-14pt,black] {\small $3$};
\end{scope}
\end{tikzpicture}
\end{center}
and we arrive at the composition $(4,3,2,3)$ in $\C_{>1}(12)$. This algorithm is reversible.

\section{Fibonacci colored compositions}
\label{sec:FiboColoring}

In this section, we consider the sets of Fibonacci colored compositions, $\C_{F_{k+1}}(n)$, $\C_{F_{k}}(n)$, and $\C_{F_{k-1}}(n)$. For their enumeration, we give explicit bijections between these sets and certain sets of ternary words\footnote{Words over the alphabet $\{0,1,2\}$.} and other restricted compositions. This is done by representing compositions of $n$ as tilings of an $n$-board, where each part $k$ of the composition corresponds to a block/tile of length $k$. We then use the interpretations of the Fibonacci numbers discussed in the previous section to create secondary tilings that represent the different colors of a block.

\begin{theorem}
\label{thm:noAdjacent0s}
The set $\C_{F_{k+1}}(n)$ of compositions of $n$, where part $k$ comes in $F_{k+1}$ colors, is in one-to-one correspondence with the set of ternary words of length $n-1$ having no adjacent $0$s. These sets are enumerated by the sequence 1, 3, 8, 22, 60, 164, 448,\dots, see \oeis{A155020} as well as \oeis{A028859}.
\end{theorem}
\begin{proof}
To describe the bijection between our sets, we begin by drawing the compositions of $n$ as tilings of an $n$-board, where part $k$ is colored by $F_{k+1}$. Since $F_{k+1}=|\C_{1,2}(k)|$, every color choice for part $k$ may be represented by a secondary tiling of a $k$-board using tiles of length 1 or 2. The following table shows the possible secondary tilings for parts 1 through 4:
\medskip
\begin{center}
\def\R{\rule[-1ex]{0ex}{4.5ex}}
\def\hsp{\hspace{-3pt}}
\def\lsf{0.45} 
\begin{tabular}{c|c|l}
\rule[-1.1ex]{0ex}{3.5ex} $k$ & $F_{k+1}$ & \,Secondary tilings \hspace{200pt} \\ \hline\hline
\R 1 & 1 & \parbox{3em}{\plainTile{1}{\lsf}} \\
\R 2 & 2 & \parbox{10em}{\dottedTile{2}{1}{\lsf}\hsp \plainTile{2}{\lsf}} \\
\R 3 & 3 & \parbox{20em}{\dottedTile{3}{1,2}{\lsf}\hsp \dottedTile{3}{1}{\lsf}\hsp \dottedTile{3}{2}{\lsf}} \\
\R 4 & 5 & \parbox{28em}{\dottedTile{4}{1,2,3}{\lsf}\hsp \dottedTile{4}{1,2}{\lsf}\hsp
\dottedTile{4}{1,3}{\lsf}\hsp \dottedTile{4}{2,3}{\lsf}\hsp \dottedTile{4}{2}{\lsf}}
\end{tabular}
\end{center}

\bigskip
With this in mind, our bijection is best explained with an example. Suppose $(2,4,1,3)$ is a composition of 10, where part 2 is colored by \parbox{6ex}{\plainTile{2}{0.45}}, part 4 is colored by 
\parbox{11.5ex}{\dottedTile{4}{1,2}{0.45}}, and part 3 is colored by \parbox{8.8ex}{\dottedTile{3}{2}{0.45}}. This composition may then be visualized as the tiled board
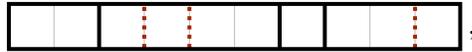
\begin{center} 
\parbox{37ex}{\mixedTile{10}{2,6,7}{3,4,9}{0.6}},
\end{center}
and we can construct a ternary word as follows. Walking the board from left to right in unit steps, write a `2' if at a solid separator, write a `1' if at a dotted separator from the secondary tiling, and write a `0' if at a step with no separator:

\medskip
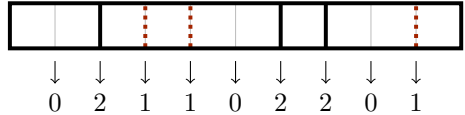
\begin{center}
\begin{tikzpicture}[scale=0.6]
\draw[lightgray] (0,0) grid (10,1);
\foreach \x in {1,...,9}{
	\draw[->, below=8pt] (\x,0) -- (\x,-0.5);
}
\foreach \x in {2,6,7}{
	\draw[Tiling] (\x,0) -- (\x,1);
	\node at (\x,-1.2) {\small $2$};
}
\foreach \x in {3,4,9}{
	\draw[dottedSep] (\x,0.025) -- (\x,1);
	\node at (\x,-1.2) {\small $1$};
}
\foreach \x in {1,5,8}{
	\node at (\x,-1.2) {\small $0$};
}
\draw[Tiling] (0,0) rectangle (10,1);
\end{tikzpicture}
\end{center}
The resulting word $021102201$ has no adjacent zeros because there are no tiles of length greater than 2 in the secondary tiling induced by the Fibonacci coloring. This algorithm defines a map that works for compositions of any size. The positions of the 2s are uniquely determined by the composition, and the positions of the 1s are determined by the coloring.

Conversely, given a ternary word of length $n-1$ with no adjacent zeros, we construct a tiled $n$-board as follows. We start by drawing our blank $n$-board and aligning the ternary word with the $n-1$ unit separators of the board. Then from left to right, we read the word drawing solid and dotted separators. At the letter $i$, we place a solid bar if $i$ is a 2, a dotted bar if $i$ is a 1, or nothing if $i$ is 0. The resulting board represents a composition of $n$ whose parts are determined by the solid bars. Moreover, each part $k$, defined between the 2s of the word, comes with a secondary tiling determined by the 1s and 0s. Since these words have no adjacent zeros, the subtiles must be of length 1 or 2. Observe that there are $F_{k+1}$ such secondary tilings. 
\end{proof}

\begin{theorem}\label{thm:Pell}
The set $\C_{F_{k}}(n)$ of compositions of $n$, where part $k$ comes in $F_{k}$ colors, is in one-to-one correspondence with the set of ternary words of length $n-1$ that avoid runs of $0$s of odd length. These sets are enumerated by the Pell numbers, 1, 2, 5, 12, 29, 70,\dots, see \oeis{A000129}.
\end{theorem}
\begin{proof}
This proof is similar to the proof of Theorem~\ref{thm:noAdjacent0s}. Once again, we draw our composition as a board of length $n$, but this time part $k$ is colored by $F_{k}=|\C_{\rm odd}(k)|$. Thus, the secondary tilings of part $k$ only use tiles of odd length. For $k\le 5$, these tilings are: 
\medskip
\begin{center}
\def\R{\rule[-1ex]{0ex}{4.5ex}}
\def\hsp{\hspace{-3pt}}
\def\lsf{0.45} 
\begin{tabular}{c|c|l}
\rule[-1.1ex]{0ex}{3.5ex} $k$ & $F_{k}$ & \,Secondary tilings \hspace{265pt} \\ \hline\hline
\R 1 & 1 & \parbox{3ex}{\plainTile{1}{\lsf}} \\
\R 2 & 1 & \parbox{6ex}{\dottedTile{2}{1}{\lsf}} \\
\R 3 & 2 & \parbox{20ex}{\dottedTile{3}{1,2}{\lsf}\hsp \plainTile{3}{\lsf}} \\
\R 4 & 3 & \parbox{40ex}{\dottedTile{4}{1,2,3}{\lsf}\hsp \dottedTile{4}{1}{\lsf}\hsp \dottedTile{4}{3}{\lsf}} \\
\R 5 & 5 & \parbox{77ex}{\dottedTile{5}{1,2,3,4}{\lsf}\hsp \dottedTile{5}{1,2}{\lsf}\hsp
\dottedTile{5}{1,4}{\lsf}\hsp \dottedTile{5}{3,4}{\lsf}\hsp \plainTile{5}{\lsf}}
\end{tabular}
\end{center}

\medskip
As before, we proceed with an example. For instance, given the colored composition $(4,2,5,1)$, where part 4
is colored by \parbox{11.3ex}{\dottedTile{4}{1}{0.45}} and part 5 is colored by \parbox{14.2ex}{\dottedTile{5}{3,4}{0.45}}, we can create a ternary word using the algorithm introduced in the proof of Theorem~\ref{thm:noAdjacent0s}:

\medskip
\begin{center}
\begin{tikzpicture}[scale=0.6]
\draw[lightgray] (0,0) grid (12,1);
\foreach \x in {1,...,11}{
	\draw[->, below=8pt] (\x,0) -- (\x,-0.5);
}
\foreach \x in {4,6,11}{
	\draw[Tiling] (\x,0) -- (\x,1);
	\node at (\x,-1.2) {\small $2$};
}
\foreach \x in {1,5,9,10}{
	\draw[dottedSep] (\x,0.025) -- (\x,1);
	\node at (\x,-1.2) {\small $1$};
}
\foreach \x in {2,3,7,8}{
	\node at (\x,-1.2) {\small $0$};
}
\draw[Tiling] (0,0) rectangle (12,1);
\end{tikzpicture}
\end{center}
The resulting ternary word $10021200112$ has zeros only in even runs since the parts of the secondary tilings have only odd length. This algorithm works for compositions of any size, and the positions of the 0s and 1s in the ternary word are uniquely determined by the coloring of the composition. The inverse map is the one described in Theorem~\ref{thm:noAdjacent0s}. Observe that if the ternary words have zeros only in even runs, then the subtiles created in the parts of the composition will all have odd length as desired.
\end{proof}

\begin{theorem} \label{thm:Jacobsthal}
The set $\C_{F_{k-1}}(n)$ of compositions of $n$, where part $k$ comes in $F_{k-1}$ colors, is in one-to-one correspondence with each of the following sets:
\begin{enumerate}[\ $(a)$]
\item compositions of $n-1$ ending with an odd part,
\item compositions of $n$ ending with an even part.
\end{enumerate}
Moreover, for $n>2$, there is a bijection between $\C_{F_{k-1}}(n)$ and the set of:
\begin{enumerate}[\ $(a)$]
\setcounter{enumi}{2}
\item ternary words of length $n-2$ in which 1 and 2 avoid runs of odd lengths,
\item ternary words of length $n-3$ with no adjacent nonzero letters.
\end{enumerate}
These sets are enumerated by the Jacobsthal numbers, see \oeis{A001045}.
\end{theorem}
\begin{proof}
We start by drawing every composition in $\C_{F_{k-1}}(n)$ as an $n$-board with secondary tilings described by $F_{k-1}=|\C_{>1}(k)|$. Thus, each part of the composition can be tiled using subtiles of length greater than 1. In particular, there are no parts of length 1, and for parts 2 through 5, we have:
\medskip
\begin{center}
\def\R{\rule[-1ex]{0ex}{4.5ex}}
\def\hsp{\hspace{-3pt}}
\def\lsf{0.45} 
\begin{tabular}{c|c|l}
\rule[-1.1ex]{0ex}{3.5ex} $k$ & $F_{k-1}$ & \,Secondary tilings \hspace{120pt} \\ \hline\hline
\R 2 & 1 & \parbox{6ex}{\plainTile{2}{\lsf}} \\
\R 3 & 1 & \parbox{10ex}{\plainTile{3}{\lsf}} \\
\R 4 & 2 & \parbox{30ex}{\dottedTile{4}{2}{\lsf}\hsp \plainTile{4}{\lsf}} \\
\R 5 & 3 & \parbox{46ex}{\dottedTile{5}{2}{\lsf}\hsp \dottedTile{5}{3}{\lsf}\hsp \plainTile{5}{\lsf}}
\end{tabular}
\end{center}

\medskip
\begin{enumerate}[$(a)$]
\item 
In order to make a composition of $n-1$ that ends with an odd part, we use a similar bundling technique as in the previous section. First, within each of the secondary tiles of the given Fibonacci colored composition, we draw additional gray separators to subdivide the tiles into square tiles. Then, we bundle the $n-1$ vertical separators according to their type: 
\begin{itemize}
\item[$\circ$] gray -- does not induce bundling, 
\item[$\circ$] dotted -- induces bundling with the separator on its left, 
\item[$\circ$] solid -- induces bundling with the separators on both sides of it.
\end{itemize}
Finally, we remove the frame of the board, and the bundled separators are then gathered to create the parts of the new composition.

For example, given the colored composition $(6,2,5,4)$ of 17, where part 6 is colored by \parbox{17ex}{\dottedTile{6}{2,4}{0.45}}, part 5 is colored by \parbox{14.2ex}{\dottedTile{5}{3}{0.45}}, and part 4 is colored by \parbox{11.5ex}{\dottedTile{4}{2}{0.45}}, the above algorithm gives:

\medskip
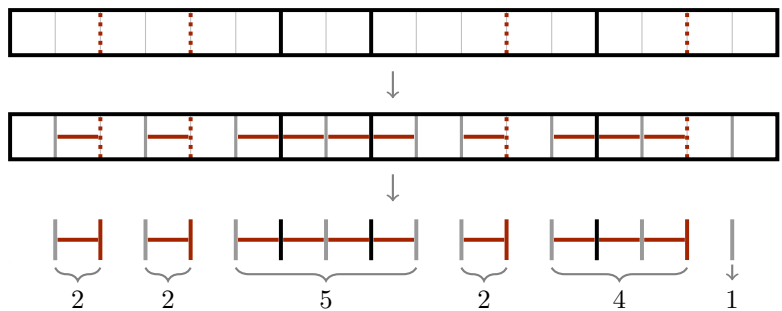
\begin{center}
\mixedTile{17}{6,8,13}{2,4,11,15}{0.6}
\begin{tikzpicture}[scale=0.6]
\begin{scope}[xshift=.53\textwidth,yshift=75]
\draw[white,thin] (0,1.85)--(0,1.85);
\draw[thick,->,gray] (0,1.6) -- (0,1);
\end{scope}

\def\delta{0.05}
\begin{scope}[yshift=65]
\draw[lightgray] (0,0) grid (17,1);
\foreach \x in {6,8,13}{
\draw[Tiling] (\x,0) -- (\x,1);
}
\foreach \x in {1,3,5,7,9,10,12,14,16}{
    \draw[very thick, gray!80] (\x,0) -- (\x,1);
}
\foreach \x in {1,3,5,6,7,8,10,12,13,14}{
\draw[solidBand] (\x+\delta,0.5) -- (\x+1-\delta,0.5);
}
\foreach \x in {2,4,11,15}{
	\draw[dottedSep] (\x,0.025) -- (\x,1);
}
\draw[Tiling] (0,0) rectangle (17,1);
\end{scope}

\begin{scope}[xshift=.53\textwidth,yshift=10]
\draw[thick,->,gray] (0,1.6) -- (0,1);
\end{scope}

\begin{scope}
\draw[lightgray] (0,0) grid (17,1);
\foreach \x in {6,8,13}{
\draw[Tiling] (\x,0) -- (\x,1);
}
\foreach \x in {1,3,5,6,7,8,10,12,13,14}{
\draw[solidBand] (\x+\delta,0.5) -- (\x+1-\delta,0.5);
}
\foreach \x in {1,3,5,7,9,10,12,14,16}{
    \draw[Tiling,gray!80] (\x,0) -- (\x,1);
}
\foreach \x in {2,4,11,15}{
	\draw[solidSep] (\x,\delta) -- (\x,1);
}

\draw[Remover] (0,0) -- (17,0) -- (17,1) -- (0,1) -- (0,0);

\foreach \x in {1,3,10}{
\draw [decorate,decoration={brace,amplitude=5pt,mirror,raise=2pt},thick,gray]
 (\x,0) -- (\x+1,0) node [midway,yshift=-14pt,black] {\small $2$};
}
\draw [decorate,decoration={brace,amplitude=5pt,mirror,raise=2pt},thick,gray]
 (5,0) -- (9,0) node [midway,yshift=-14pt,black] {\small $5$};
\draw [decorate,decoration={brace,amplitude=5pt,mirror,raise=2pt},thick,gray]
 (12,0) -- (15,0) node [midway,yshift=-14pt,black] {\small $4$};
\draw[->, thick, gray] (16,-0.05) -- (16,-0.4);
\node at (16,0) [yshift=-14pt,black] {\small $1$};
\end{scope}
\end{tikzpicture}
\end{center}

The resulting composition $(2,2,5,2,4,1)$ of 16 has its parts depicted by the bundles created. Since the starting colored composition has no parts or subtiles of length 1, its last separator will always be a gray separator. Therefore, the constructed composition will never end in an even part.

Conversely, from a composition of length $n-1$ ending with an odd part, we can create a composition of $n$ colored by $F_{k-1}$ as follows. First, draw $n-1$ vertical lines (gray separators) and represent the parts of the composition by creating bundles of the appropriate size. Then, within each odd part, convert any even positioned vertical line into a solid separator. Do the same within each of the even parts except for the last separator, which should be converted into a dotted separator. Finally, add solid lines at the beginning and end of the representation (independently from the parts), remove the bundling, and add horizontal lines at the top and bottom to create a board of length $n$. The resulting tiled board represents an element of $\C_{F_{k-1}}(n)$. 

For example, for the composition $(4,2,3,1,1,5)$ of $16$, we get:

\medskip
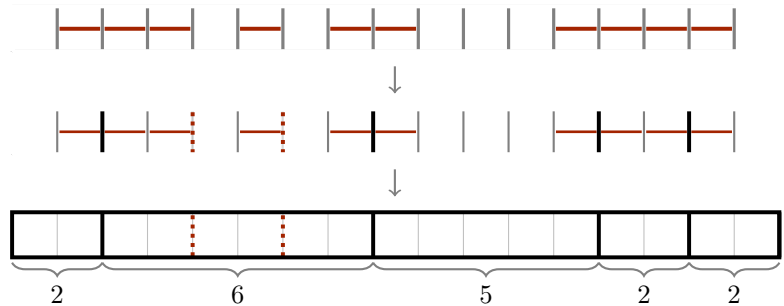
\begin{center}
\begin{tikzpicture}[scale=0.6]
\def\delta{0.05}
\begin{scope}[yshift=130]
\draw[very thick, gray] (0,0) grid (16,1);
\foreach \x in {1,2,3,5,7,8,12,13,14,15}{
\draw[solidBand] (\x+\delta,0.5) -- (\x+1-\delta,0.5);
}
\draw[Remover] (0,0) -- (17,0) -- (17,1) -- (0,1) -- (0,0);
\end{scope}

\begin{scope}[xshift=.53\textwidth,yshift=75]
\draw[thick,->,gray] (0,1.6) -- (0,1);
\end{scope}

\begin{scope}[yshift=65]
\draw[thick, gray] (0,0) grid (16,1);
\foreach \x in {1,2,3,5,7,8,12,13,14,15}{
\draw[solidBand, line width=1] (\x+\delta,0.5) -- (\x+1-\delta,0.5);
}
\foreach \x in {4,6}{
	\draw[thick,white] (\x,0) -- (\x,1);
	\draw[lightgray] (\x,0.05) -- (\x,0.95);
	\draw[dottedSep] (\x,0.025) -- (\x,1);
}
\foreach \x in {2,8,13,15}{
\draw[Tiling] (\x,0) -- (\x,1);
}
\draw[Remover] (0,0) -- (17,0) -- (17,1) -- (0,1) -- (0,0);
\end{scope}

\begin{scope}[xshift=.53\textwidth,yshift=10]
\draw[thick,->,gray] (0,1.6) -- (0,1);
\end{scope}

\begin{scope}
\draw[lightgray] (0,0) grid (17,1);
\foreach \x in {4,6}{
	\draw[dottedSep] (\x,0.025) -- (\x,1);
}
\foreach \x in {2,8,13,15,17}{
\draw[Tiling] (\x,0) -- (\x,1);
}
\draw[Tiling] (0,0) rectangle (17,1);

\foreach \x in {0,13,15}{
\draw [decorate,decoration={brace,amplitude=5pt,mirror,raise=2pt},thick,gray]
 (\x,0) -- (\x+2,0) node [midway,yshift=-14pt,black] {\small $2$};
}
\draw [decorate,decoration={brace,amplitude=5pt,mirror,raise=2pt},thick,gray]
 (8,0) -- (13,0) node [midway,yshift=-14pt,black] {\small $5$};
 \draw [decorate,decoration={brace,amplitude=5pt,mirror,raise=2pt},thick,gray]
 (2,0) -- (8,0) node [midway,yshift=-14pt,black] {\small $6$};
\end{scope}
\end{tikzpicture}
\end{center}

The final composition is the Fibonacci colored composition $(2,6,5,2,2)$ of 17, where part 6 is colored by \parbox{16.7ex}{\dottedTile{6}{2,4}{0.45}} and part 5 is colored by \parbox{14.2ex}{\plainTile{5}{0.45}}. Observe that in the resulting composition, we never create a part or secondary tile of length $1$ due to the defined placement of dotted and solid separators.
   
\item
By adding 1 to the last part, we get a natural bijective map between compositions of $n-1$ ending with an odd part and compositions of $n$ ending with an even part.

\item
The claimed bijection is easily shown using the tiled board representation discussed at the beginning of this proof. Given a tiled board corresponding to a composition in $\C_{F_{k-1}}(n)$, write a '2' on the unit tiles immediately to the left and right of every solid separator, and write a '1' on the unit tiles immediately to the left and right of every dotted separator. Then, excluding the end unit tiles of the board, fill any empty spaces with a '0'. Since the solid and dotted separators lead to pairs of 2s and 1s, respectively, this process clearly creates a ternary word of length $n-2$ in which 1 and 2 avoid runs of odd lengths. 

For example, the colored composition $(6,2,5,4)$ of 17 considered in part (a) gives:

\medskip
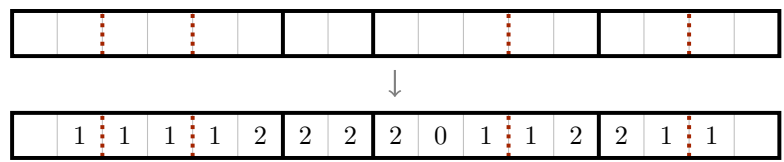
\begin{center}
\mixedTile{17}{6,8,13}{2,4,11,15}{0.6} \\
\begin{tikzpicture}[scale=0.6]
\begin{scope}[xshift=.53\textwidth, yshift=75]
\draw[white,thin] (0,1.8)--(0,1.8);
\draw[thick,->,gray] (0,1.6) -- (0,1);
\end{scope}

\begin{scope}[yshift=65]
\draw[lightgray] (0,0) grid (17,1);
\foreach \x in {6,8,13}{
	\draw[Tiling] (\x,0) -- (\x,1);
	\node at (\x-0.5,0.5) {\small $2$};
	\node at (\x+0.5,0.5) {\small $2$};
}
\foreach \x in {2,4,11,15}{
	\draw[dottedSep] (\x,0.025) -- (\x,1);
	\node at (\x-0.5,0.5) {\small $1$};
	\node at (\x+0.5,0.5) {\small $1$};
}
\foreach \x in {9}{
	\node at (\x+0.5,0.5) {\small $0$};
}
\draw[Tiling] (0,0) rectangle (17,1);
\end{scope}
\end{tikzpicture}
\end{center}
and we obtain the ternary word $111122220112211$ of length $15$.

The inverse map is straightforward; we just need to be diligent when placing separators to create the colored composition. Once a separator is placed between a pair of 1s or 2s, they cannot be paired again, so $2|2|2|2$ cannot happen; the correct way to insert separators would be $2|2 \; 2|2$. This, combined with the placement of 1s, is the reason why the created compositions will not contain any tiles of length $1$.

\item
Lastly, using the algorithm in the proof of Theorem~\ref{thm:noAdjacent0s} that maps Fibonacci colored compositions into ternary words, every element of $\C_{F_{k-1}}(n)$ uniquely corresponds to a ternary word of length $n-1$ with no adjacent nonzero letters and such that the first and last letters are 0. Removing these two 0s, we get a ternary word of length $n-3$ with the desired properties.
\end{enumerate}
\end{proof}

\section{Even-indexed Fibonacci colored compositions}
\label{sec:evenFibo}

We now turn our attention to compositions colored by the even-indexed Fibonacci numbers $1, 3, 8, 21, 55,\dots$, \oeis{A001906}. As a tool, we use the spotted tilings introduced by Hopkins~\cite{BHop12}. A {\em spotted tile of size} $j$ is a $1\times j$ rectangular tile with one spot at position $i$, where $1 \leq i \leq j$. Thus, there are exactly $j$ distinct spotted tiles of size $j$, and the number of spotted tilings of a $k$-board is the Fibonacci number $F_{2k}$. For example, there are only three tiles of size 3,

\smallskip
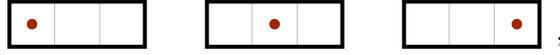
\def\locsf{0.6}
\begin{center}
 \spottedTile{3}{}{}{1}{\locsf}\quad \spottedTile{3}{}{}{2}{\locsf}\quad \spottedTile{3}{}{}{3}{\locsf} \parbox{1ex}{\!,}
\end{center}
and if we consider all spotted tilings of a $3$-board, we need to add the tilings
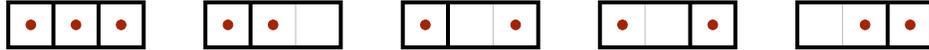
\[
\spottedTile{3}{1,2}{}{1,2,3}{\locsf}\qquad
\spottedTile{3}{1}{}{1,2}{\locsf}\qquad
\spottedTile{3}{1}{}{1,3}{\locsf}\qquad
\spottedTile{3}{2}{}{1,3}{\locsf}\qquad
\spottedTile{3}{2}{}{2,3}{\locsf}
\]
for a total of eight ($F_6$) spotted tilings of length 3.

Let $\C_{\rm spot}(k)$ denote the set of colored compositions of $k$ represented by the spotted tilings of length $k$. As shown in \cite{BHop12}, $|\C_{\rm spot}(k)| = F_{2k}$. We will use this interpretation of the even-indexed Fibonacci numbers to represent our colors in this section.

Let $L_n$ be the set of spanning trees\footnote{A spanning tree of a graph on $m$ vertices is a subset of $m-1$ edges that form a tree.} in the $n\times 2$ grid graph, also known as the $n$-Ladder graph. These trees connect the $2n$ vertices of the graph using only $2n-1$ edges with no cycles. 

For example, for $n=1,2$ we have 1 and 4 spanning trees, respectively,
\medskip
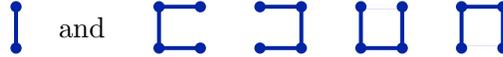
\def\locsf{0.55}
\begin{center}
\spanTree{0}{}{}{1}{\locsf} \;\tikz[baseline=-10]{\node{and}}
\spanTree{1}{}{}{1}{\locsf}
\spanTree{1}{}{}{2}{\locsf}
\spanTree{1}{}{1}{1,2}{\locsf}
\spanTree{1}{1}{}{1,2}{\locsf}
\end{center}
For $n=3$, there are 15 such spanning trees:
\medskip
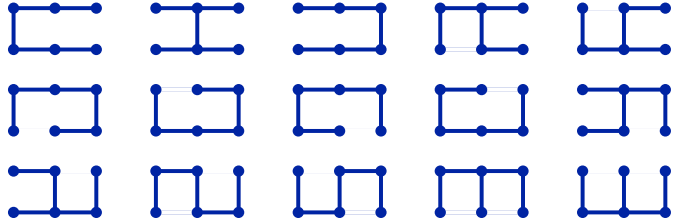
\begin{center}
\spanTree{2}{}{}{1}{\locsf}
\spanTree{2}{}{}{2}{\locsf}
\spanTree{2}{}{}{3}{\locsf}
\spanTree{2}{1}{}{1,2}{\locsf}
\spanTree{2}{}{1}{1,2}{\locsf}\\[10pt]
\spanTree{2}{1}{}{1,3}{\locsf}
\spanTree{2}{}{1}{1,3}{\locsf}
\spanTree{2}{2}{}{1,3}{\locsf}
\spanTree{2}{}{2}{1,3}{\locsf}
\spanTree{2}{2}{}{2,3}{\locsf}\\[10pt]
\spanTree{2}{}{2}{2,3}{\locsf}
\spanTree{2}{1}{2}{1,2,3}{\locsf}
\spanTree{2}{2}{1}{1,2,3}{\locsf}
\spanTree{2}{1,2}{}{1,2,3}{\locsf}
\spanTree{2}{}{1,2}{1,2,3}{\locsf}
\end{center}

\medskip
The following theorem is the main result of this section.
\begin{theorem}
The set $\C_{F_{2k}}(n)$ of compositions of $n$, where part $k$ comes in $F_{2k}$ colors, is in one-to-one correspondence with the set of spanning trees in $L_n$, and they are also equinumerous with the set of $01$-avoiding words of length $n-1$ over the alphabet $\{0,1,2,3\}$. These sets are enumerated by the sequence 1, 4, 15, 56, 209, 780, 2911, 10864,\dots, see \oeis{A001353}.
\end{theorem}

\begin{proof}
We will use the spotted tilings, $\C_{\rm spot}(k)$, as secondary tilings of our compositions. The bijection between $\C_{F_{2k}}(n)$ and $L_n$ is given by the following algorithm:

\begin{enumerate}[(i)]
\item Given an element of $\C_{F_{2k}}(n)$, represent it as a tiling of an $n$-board with its parts colored by a secondary spotted tiling. 
\item Superimpose an $n$-Ladder graph over the given composition so that their centers coincide.
\item Remove horizontal edges of the graph as follows: If a square subgraph of the grid contains a solid separator of the composition, remove the bottom edge of that square, and if it contains a dotted separator of the composition, remove the top edge of that square.
\item Vertical edges of the graph are then removed based on the positions of the spotted tiles in the composition. If an edge intersects a spot, it is kept, and all other vertical edges are then removed. The remaining edges make a spanning tree in $L_n$.
\end{enumerate}

To illustrate our map, we proceed with an example. Given the composition $(2,1,5,2)$ of 10, where the first part $2$ is colored by \parbox{6ex}{\spottedTile{2}{}{1}{1,2}{0.45}}, part $5$ is colored by \parbox{14.2ex}{\spottedTile{5}{}{4}{2,5}{0.45}}, and the last part $2$ is colored by \parbox{6ex}{\spottedTile{2}{}{}{1}{0.45}}, we follow the above steps to create a matching spanning tree in $L_{10}$: 

\medskip
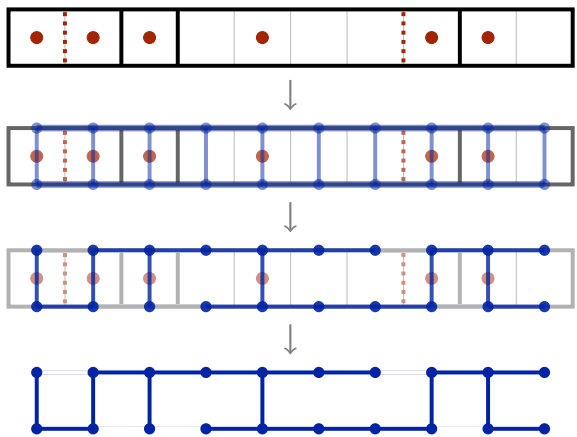
\begin{center}
\hskip-3.5pt
\spottedTile{10}{2,3,8}{1,7}{1,2,3,5,8,9}{0.75}

\tikz{\node[white] at (0,0.4){}; \node[white] at (0,0){}; \draw[thick,->,gray] (0,0.4) -- (0,0);}

\begin{tikzpicture}[scale=0.75]
\draw[lightgray, thin] (0,0) grid (10,1);
\foreach \x in {2,3,8}{
	\draw[Tiling, black!60] (\x,0) -- (\x,1);
}
\foreach \x in {1,7}{
	\draw[dottedSep, tileColor!70] (\x,0.06) -- (\x,1);
}
\foreach \x in {1,2,3,5,8,9}{
	\draw[color=spotColor!70, fill=spotColor!70] (\x-0.5,0.5) circle (3pt); 
}
\draw[Tiling, black!60] (0,0) rectangle (10,1);
\draw[Tree, line width=2.5, opacity=0.5] (0.5,1) -- (9.5,1);
\draw[Tree, line width=2.5, opacity=0.5] (0.5,0) -- (9.5,0);
\foreach \x in {0,...,9}{
	\draw[Tree, opacity=0.5] (\x+0.5,0) -- (\x+0.5,1); 
	\node[node,opacity=0.5] at (\x+0.5,0) {};
	\node[node,opacity=0.5] at (\x+0.5,1) {};
 }
\end{tikzpicture}

\tikz{\node[white] at (0,0.4){}; \node[white] at (0,0){}; \draw[thick,->,gray] (0,0.4) -- (0,0);}

\begin{tikzpicture}[scale=0.75]
\draw[lightgray, very thin] (0,0) grid (10,1);
\draw[Tiling,black!30] (0,0) rectangle (10,1);
\draw[Tree,opacity=0.8] (0.5,0) -- (9.5,0);
\draw[Tree,opacity=0.8] (0.5,1) -- (9.5,1);
\foreach \x in {2,3,8}{
	\draw[Tiling, black!30] (\x,0) -- (\x,0.959);
	\draw[spanRem] (\x-0.5,0) -- (\x+0.5,0);
	\draw[line width=1.5, black!30] (\x-0.5,0) -- (\x+0.5,0);
}
\foreach \x in {1,7}{
	\draw[dottedSep, tileColor!50] (\x,0.06) -- (\x,1);
	\draw[spanRem] (\x-0.5,1) -- (\x+0.5,1);
	\draw[line width=1.5, black!30] (\x-0.5,1) -- (\x+0.5,1);
}
\foreach \x in {1,2,3,5,8,9}{
	\draw[color=spotColor!50,fill=spotColor!50] (\x-0.5,0.5) circle (3pt); 
	\draw[Tree,opacity=0.8] (\x-0.5,0) -- (\x-0.5,1); 
}
\foreach \x in {0,...,9}{ 
	\node[node,opacity=0.9] at (\x+0.5,0) {};
	\node[node,opacity=0.9] at (\x+0.5,1) {};
 }
\end{tikzpicture}

\tikz{\node[white] at (0,0.4){}; \node[white] at (0,0){}; \draw[thick,->,gray] (0,0.4) -- (0,0);}

\spanTree{9}{2,3,8}{1,7}{1,2,3,5,8,9}{0.75}
\end{center}

In general, the process starts by superimposing an $n$-Ladder graph (having $3n-2$ edges) on top of the given composition. If the composition has $k$ dots, then the $n$-board representing it must have $k$ tiles, including its secondary tiling, and therefore, it will have $k-1$ separators -- solid and dotted combined. Then, vertical edges are removed based on the absence of spots within the coloring. As there are $n-k$ spaces without a spot, the process will remove $n-k$ edges. In addition, horizontal edges are removed based on the separators of the composition, so this part of the algorithm removes $k-1$ more edges. Thus, from the original $3n-2$ edges, we end up removing $(n-k)+(k-1) = n-1$ edges, leaving the remaining graph with $2n-1$ edges that connect all of its $2n$ vertices. Note that when two adjacent horizontal edges are removed, there must be a spot between the two separators within the composition. Therefore, the incident vertex is connected to the graph by a vertical edge. Also, since there may not be two spots in a single tile, the resulting graph cannot have cycles. In conclusion, our algorithm always provides a spanning tree in $L_n$.

The inverse map is clear from the process described above. Once the spanning tree is drawn, an $n$-board is superimposed on the tree with their centers aligned. Then, for any opening along the top of the tree, a dotted separator is placed on the $n$-board within the corresponding unit square. Similarly, for any opening along the bottom of the tree, a solid separator is placed within the unit square. Finally, dots are placed on all unit squares that intersect a vertical edge of the tree. Removing the tree, we are left with a composition of $n$ colored by spotted tilings.

\medskip
The bijection to the set of $01$-avoiding words over the alphabet $\{0,1,2,3\}$ relies on a simple algorithm.
Given an element of $\C_{F_{2k}}(n)$, we represent it as a tiling of an $n$-board with its parts colored by spotted tilings. Then, walking the board from left to right in unit steps, we write a `3' if at a solid separator and a `2' if at a dotted separator from the secondary tiling.  Steps with no separator happen inside a spotted tile of length greater than one.  Within each spotted tile, we write `1' if at a step to the left of the spot and `0' if at a step to the right of the spot. Clearly, any words created this way will not contain the subword $01$.

For example,

\smallskip
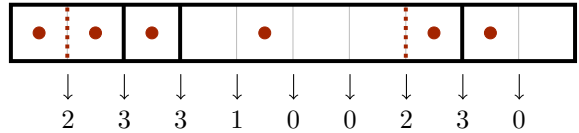
\begin{center}
\hskip-4pt \spottedTile{10}{2,3,8}{1,7}{1,2,3,5,8,9}{0.75}
\vskip3pt
\begin{tikzpicture}[scale=0.75]
\foreach \x in {1,...,9}{\draw[->] (\x,0) -- (\x,-0.45);}
\foreach \x in {2,3,8}{
	\node[below=10pt] at (\x,0) {\small $3$};
}
\foreach \x in {1,7}{
	\node[below=10pt] at (\x,0) {\small $2$};
}
\foreach \x in {5,6,9}{
	\node[below=10pt] at (\x,0) {\small $0$};
}
\node[below=10pt] at (4,0) {\small $1$};
\end{tikzpicture}
\end{center}
In other words, the composition $(2,1,5,2)$ of 10, colored as depicted above, corresponds to the word $233100230$ of length 9.
\end{proof}

\begin{remark}
The sequence \oeis{A001353} is also known to count the number of $2\times n$ simple rectangular mazes. With no intention of going into details, here is an example that illustrates how such rectangular mazes are related to the elements of $\C_{F_{2k}}(n)$ and the various objects described in the previous theorem.

\medskip
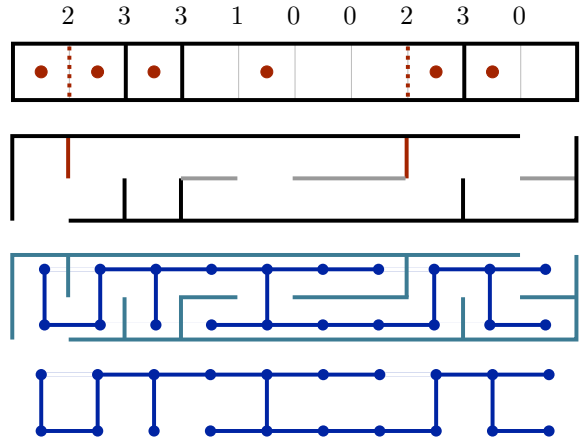
\begin{center}
\def\y{2.7}
\begin{tikzpicture}[scale=0.75]
\begin{scope}[yshift=80]
\foreach \x in {-2.85,-1.85,3.15}{\node[above=10pt] at (\x,\y) {\small $3$};}
\foreach \x in {-3.85,2.15}{\node[above=10pt] at (\x,\y) {\small $2$};}
\foreach \x in {0.15,1.15,4.15}{\node[above=10pt] at (\x,\y) {\small $0$};}
\node[above=10pt] at (-0.85,\y) {\small $1$};
\node at (0,2.5) {\spottedTile{10}{2,3,8}{1,7}{1,2,3,5,8,9}{0.75}};
\end{scope}
\begin{scope}[yshift=76]
\foreach \x in {-3.85,2.15}{\draw[solidBand] (\x,1.5) -- (\x,0.75);}
\foreach \x in {-2.85,-1.85,3.15}{\draw[Tiling] (\x,0) -- (\x,0.75);}
\draw[Tiling,gray!80] (-1.85,0.75) -- (-0.85,0.75);
\draw[Tiling,gray!80] (0.13,0.75) -- (2.13,0.75);
\draw[Tiling,gray!80] (4.16,0.75) -- (5.16,0.75);
\draw[Tiling] (-4.84,0) -- (-4.84,1.5) -- (4.16,1.5);
\draw[Tiling] (-3.84,0) -- (5.16,0) -- (5.16,1.5);
\end{scope}
\def\maze{cyan!50!black}
\begin{scope}[yshift=16]
\node at (0,0.75) {\spanTree{9}{2,3,8}{1,7}{1,2,3,5,8,9}{0.74}};
\foreach \x in {-3.85,2.15}{\draw[Tiling,\maze] (\x,1.5) -- (\x,0.75);}
\foreach \x in {-2.85,-1.85,3.15}{\draw[Tiling,\maze] (\x,0) -- (\x,0.75);}
\draw[Tiling,\maze] (-1.885,0.75) -- (-0.85,0.75);
\draw[Tiling,\maze] (0.13,0.75) -- (2.185,0.75);
\draw[Tiling,\maze] (4.16,0.75) -- (5.16,0.75);
\draw[Tiling,\maze] (-4.84,0) -- (-4.84,1.5) -- (4.16,1.5);
\draw[Tiling,\maze] (-3.84,0) -- (5.16,0) -- (5.16,1.5);
\end{scope}
\begin{scope}[yshift=-16]
\node at (0,0) {\spanTree{9}{2,3,8}{1,7}{1,2,3,5,8,9}{0.75}};
\end{scope}
\end{tikzpicture}
\end{center}
\end{remark}

\section{Odd-indexed Fibonacci colored compositions}
\label{sec:oddFibo}

A partition of the set $[n]=\{1,\dots,n\}$ is a set of disjoint nonempty sets (called blocks) whose union is $[n]$. Every partition $\pi$ of $[n]$ can be represented by an arc digram obtained by drawing an arc between each pair of integers that appear consecutively in the same block of $\pi$. For example, the partition $\pi = \{\{1,4\}, \{2,3,6\}, \{5\}, \{7,8\}\}$ (also written as $14\vert 236\vert 5\vert 78$) corresponds to

\medskip
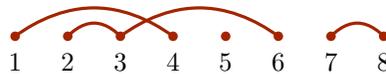
\begin{center}
\drawarc{0.7}{8}{1/4,2/3,3/6,7/8} \\
\tikz{
\foreach \i in {1,...,8}{
 \node at (0.7*\i,0) {\small \i};
}}
\end{center}
Two arcs $(i_1,j_1)$ and $(i_2,j_2)$ make a {\em crossing} if $i_1<i_2<j_1<j_2$, and they make a {\em nesting} if $i_1<i_2<j_2<j_1$. A {\em noncrossing/nonnesting partition} is a partition with no crossings/nestings. Some examples are shown in Table~\ref{tab:partition_samples}. The set of noncrossing partitions of $[n]$ is equinumerous with the set of nonnesting partitions of $[n]$, both counted by the Catalan numbers \oeis{A000108}. 

\medskip
\begin{table}[ht]
\def\ff{0.55}
\small
\begin{tabular}{c@{\hskip9pt}|@{\hskip9pt}c@{\hskip12pt}|@{\hskip6pt}c@{\hskip12pt}}
Nonnesting & Noncrossing & Nonnesting \& noncrossing \\[5pt]
\rule[0ex]{0ex}{3.5ex}
\drawarc{0.7}{6}{1/4,3/5,5/6} 
& \drawarc{0.7}{6}{1/4,2/3,4/6}
& \drawarc{0.7}{6}{1/4,4/5,5/6}
\end{tabular}
\bigskip
\caption{Arc diagrams for the partitions $14\vert 2\vert 356$, $146\vert 23\vert 5$, and $1456\vert 2\vert 3$.}
\label{tab:partition_samples}
\end{table}

A partition of $[n]$ is said to be {\em indecomposable} if no subset of its blocks is a partition of $[k]$ with $k<n$. In other words, an indecomposable partition is one whose corresponding arc diagram cannot be separated into two disjoint arc diagrams with consecutive nodes. For example, the partitions in Table~\ref{tab:partition_samples} are all indecomposable, but our first example $14\vert 236\vert 5\vert 78$ is decomposable because $14\vert 236\vert 5$ is a partition of $\{1,\dots,6\}$.

Let $\P_{\rm ncn}(n)$ be the set of partitions of $[n]$ that are both noncrossing and nonnesting, and let $\P_{\rm ncn}^{i}(n)$ be the subset of such partitions that are indecomposable. Table~\ref{tab:P_ncni} shows a graphic representation of the elements of $\P_{\rm ncn}^{i}(n)$ for $n=1,\dots,5$.

\begin{table}[ht]
\def\ff{0.55}
\small
\begin{tabular}{c@{\hskip9pt}|@{\hskip9pt}c@{\hskip12pt}|@{\hskip6pt}c@{\hskip12pt}|@{\hskip6pt}c@{\hskip12pt}|@{\hskip6pt}c@{\hskip10pt}c}
\rule[-1ex]{0ex}{3ex}
$n=1$ & $n=2$ & $n=3$ & $n=4$ & \multicolumn{2}{c}{$n=5$} \\ \hline
\rule[0ex]{0ex}{3.5ex}
\tikz[scale=\ff]{\draw[fill,arc1] (0,0) circle (0.06);} & \drawarc{\ff}{2}{1/2} & \drawarc{\ff}{3}{1/2,2/3} & \drawarc{\ff}{4}{1/2,2/3,3/4} & \drawarc{\ff}{5}{1/2,2/3,3/4,4/5} & \drawarc{\ff}{5}{1/4,4/5} \\[1pt]
\hskip4pt $1$ &\;\;$12$ &\;\;$123$ &\;\;$1234$ &\;\;$12345$ &\;\; $145\vert 2\vert 3$ \\[4pt]
&& \drawarc{\ff}{3}{1/3} & \drawarc{\ff}{4}{1/3,3/4} & \drawarc{\ff}{5}{1/3,3/4,4/5} & \drawarc{\ff}{5}{1/2,2/5} \\[1pt]
&&\;\;$13\vert 2$ &\;\;$134\vert 2$ &\;\;$1345\vert 2$ &\;\;$125\vert 3\vert 4$ \\[4pt]
&&& \drawarc{\ff}{4}{1/2,2/4} & \drawarc{\ff}{5}{1/2,2/4,4/5} & \drawarc{\ff}{5}{1/5} \\[1pt]
&&&\;\;$124\vert 3$ &\;\;$1245\vert 3$ &\;\;$15\vert 2\vert 3\vert 4$ \\[4pt]
&&& \drawarc{\ff}{4}{1/4} & \drawarc{\ff}{5}{1/2,2/3,3/5} & \drawarc{\ff}{5}{1/3,3/5} \\[1pt]
&&&\;\;$14\vert 2\vert 3$ &\;\;$1235\vert 4$ &\;\;$135\vert 2\vert 4$
\end{tabular}
\bigskip
\caption{Elements of $\P_{\rm ncn}^{i}(n)$ with their arc representation.}
\label{tab:P_ncni}
\end{table}

\begin{lemma}
If $\NCNi(n)=\big|\P_{\rm ncn}^{i}(n)\big|$, then
\begin{gather*}
 \NCNi(1)=1, \;\; \NCNi(2)=1, \text{ and} \\
 \NCNi(n) = 2\cdot\NCNi(n-1) \;\text{ for } n>2.
\end{gather*}
Thus $\sum_{n=1}^\infty \NCNi(n) x^n = \frac{x-x^2}{1-2x} = x+x^2+2x^3+4x^4+8x^5+\cdots$, cf.~\cite[A011782]{Sloane}.
\end{lemma}
\begin{proof}
For $n=1,2$, we have $\P_{\rm ncn}^{i}(1)=\{\{1\}\}$ and $\P_{\rm ncn}^{i}(2)=\{\{1,2\}\}$. For $n>2$, the set $\P_{\rm ncn}^{i}(n)$ can be split into two disjoint sets depending on whether or not $n-1$ is a singleton. 

Let $A_n$ be the set of such partitions that have $n-1$ as a singleton. By removing $n-1$ and replacing $n$ by $n-1$, we get a bijective map from $A_n$ to $\P_{\rm ncn}^{i}(n-1)$. 

Now let $A'_n$ be the subset of $\P_{\rm ncn}^{i}(n)$ consisting of partitions where $n-1$ is not a singleton. In such a partition, 1, $n-1$, and $n$ must all be in the same block. If we remove $n$, we get a unique element of $\P_{\rm ncn}^{i}(n-1)$. This gives a bijection between $A'_n$ and $\P_{\rm ncn}^{i}(n-1)$. 

Since $\P_{\rm ncn}^{i}(n) = A_n\cup A'_n$ and $A_n\cap A'_n=\varnothing$, we get the claimed recurrence relation.
\end{proof}

\begin{remark}
Every partition of $[n]$ can be obtained by `concatenating' indecomposable ones. Thus, the set of partitions of $[n]$ that are both noncrossing and nonnesting is equinumerous with the set of compositions of $n$ where part $k$ comes in $\NCNi(k)$ colors. If we let $\NCN(n)=\big|\P_{\rm ncn}(n)\big|$, this means that $\{\NCN(n)\}_{n\in\mathbb{N}}$ is the {\sc invert} transform of $\{\NCNi(n)\}_{n\in\mathbb{N}}$, so
\begin{equation*}
  \sum_{n=1}^\infty \NCN(n) x^n =\frac{\frac{x-x^2}{1-2x}}{1-\frac{x-x^2}{1-2x}} = \frac{x-x^2}{1-3x+x^2}.
\end{equation*}
This is the generating function of the sequence $\{F_{2n-1}\}_{n\in\mathbb{N}}$ of odd-indexed Fibonacci numbers $1, 2, 5, 13, 34,\dots$, \oeis{A001519}. See also Marberg~\cite[Example~4.2]{Mar13}.
\end{remark}

Before going into compositions colored by the odd-indexed Fibonacci numbers, we first introduce a different class of noncrossing partitions more suitable for our purposes.

Let $\pi$ be a partition of $[n]$.  A block $B_i=\{i_1,\dots,i_r\}\subset\pi$ is said to be nested by the block $B_j=\{j_1,\dots,j_s\}\subset\pi$ if there is an index $\ell$ such that 
\[ j_\ell < i_1<\cdots< i_r < j_{\ell+1}. \] 
Using the standard arc representation of $\pi$, this means that $B_i$ lies underneath an arc of $B_j$. 

We say that $\pi=\{B_1,\dots,B_k\}$ is {\em totally nested} if either $k=1$, or for every $1<i\le k$, the block $B_i$ is nested by $B_{i-1}$. Observe that a totally nested partition is necessarily indecomposable and noncrossing. We let $\P_{\rm tn}(n)$ denote the set of partitions of $[n]$ that are totally nested. 

\begin{theorem}
The sets $\P_{\rm ncn}^{i}(n)$ and $\P_{\rm tn}(n)$ are equinumerous.
\end{theorem}
\begin{proof}
We will provide a bijective map $\varphi:\P_{\rm ncn}^{i}(n)\to \P_{\rm tn}(n)$.

Let $\pi\in\P_{\rm ncn}^{i}(n)$. If $\pi$ is a one block partition, we let $\varphi(\pi)=\pi$. Otherwise, $\pi$ consists of a first block $B_1$ that must contain the elements $1$ and $n$, together with other blocks that are all singletons. In that case, we define the map $\varphi$ as follows:

\begin{enumerate}[(i)]
\item Write the block $B_1$ as a disjoint union of sets
\[ B_1 = R_1\cup\cdots\cup R_k \]
such that every $R_j$ consists of a singleton or a sequence of consecutive numbers with 
\[ \min(R_{j+1}) - \max(R_{j})>1 \text{ for every } j\in\{1,\dots,k-1\}. \]
\item Let $S_1,\dots,S_{k-1}$ be the gaps between the $R_j$'s, that is, for every $i\in\{1,\dots,k-1\}$ let
\[ S_i = \{\max(R_i)+1,\dots,\min(R_{i+1})-1\}. \]
The set $S_1\cup\cdots\cup S_{k-1}$ is the union of the singletons of $\pi$. Thus, we get a partition of $[n]$,
\[ R_1<S_1<R_2<S_2<\cdots<R_{k-1}<S_{k-1}<R_k, \]
where $A<B$ means that $a<b$ for every $a\in A$ and $b\in B$.
\item Finally, define $\varphi(\pi)$ by connecting pairs of the above sets in the only way they can create a sequence of totally nested blocks. In other words,
\[ \varphi(\pi) = \{R_1\cup R_k,\ S_1\cup S_{k-1},\ R_2\cup R_{k-1}, \dots\}. \]
Figure~\ref{fig:varphi} shows a visualization of this step.

\begin{figure}[ht!] 
\def\R{\rule[-0.3ex]{0ex}{2ex}}
\tikzstyle{block}=[rectangle, draw, rounded corners, line width=1.2]
\begin{tikzpicture}[scale=0.8]
\draw [black, line width=1.2] (-5.5,0) parabola bend (0,1.5) (5.5,0);
\draw [arcColor, line width=1.2] (-4,0) parabola bend (0,1.17) (4,0);
\draw [black, line width=1.2] (-2.5,0) parabola bend (0,0.75) (2.5,0);
\node at (-5.5,0) [block, fill=black!15] {\scriptsize \R\;\:$R_1$\;\:};
\node at (-4,0) [block, draw=arcColor, fill=arcColor!15] {\scriptsize \R\;\:$S_1$\;\:};
\node at (-2.5,0) [block, fill=black!15] {\scriptsize \R\;\:$R_2$\;\:};
\node at (2.5,0) [block, fill=black!15] {\scriptsize \R$R_{k-1}$};
\node at (4,0) [block, draw=arcColor, fill=arcColor!15] {\scriptsize \R$S_{k-1}$};
\node at (5.5,0) [block, fill=black!15] {\scriptsize \R\;\:$R_{k}$\;\:};
\draw [fill, color=arcColor] (-1,0) circle (0.06);
\draw [fill] (0,0) circle (0.06);
\draw [fill, color=arcColor] (1,0) circle (0.06);
\end{tikzpicture}
\caption{Map $\varphi$ construction of a totally nested partition.}
\label{fig:varphi}
\end{figure}
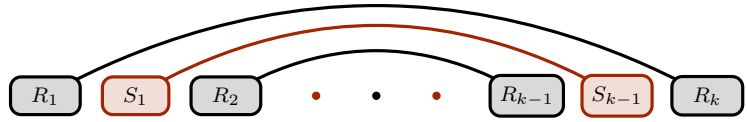
\end{enumerate}

For example, for $\pi=\{\{1,4,6,8,10,11\},\{2\},\{3\},\{5\},\{7\},\{9\}\}$, our algorithm gives
\begin{gather*} 
R_1=\{1\},\; R_2=\{4\},\; R_3=\{6\},\; R_4=\{8\},\; R_5=\{10,11\}, \\
S_1=\{2,3\},\; S_2=\{5\},\; S_3=\{7\},\; S_4=\{9\}, 
\end{gather*}
and therefore,
\begin{equation*}
\varphi(\pi) = \{\{1,10,11\},\{2,3,9\},\{4,8\},\{5,7\},\{6\}\}. 
\end{equation*}

The inverse map works as follows. Let $\sigma=\{P_1,\dots,P_k\}$ be a totally nested partition of $[n]$. If $k=1$, then $\varphi^{-1}(\sigma) = \sigma$. Otherwise, if $B$ is the union of the odd indexed blocks of $\sigma$ and $S$ is the union of the even indexed blocks of $\sigma$, then $\varphi^{-1}(\sigma)=\{B,\{i_1\},\{i_2\},\dots,\{i_s\}\}$, where $i_1,\dots,i_s$ are the elements of the set $S$. 
\end{proof}

For $n=4$, the partition $14\vert 2\vert 3$ is the only element of $\P_{\rm ncn}^{i}(4)$ that is not in $\P_{\rm tn}(4)$. Similarly, the elements of $\P_{\rm ncn}^{i}(5)$ that are not totally nested are $145\vert 2\vert 3$, $125\vert 3\vert 4$, $15\vert 2\vert 3\vert 4$, and $135\vert 2\vert 4$. These partitions get mapped by $\varphi$ as follows:
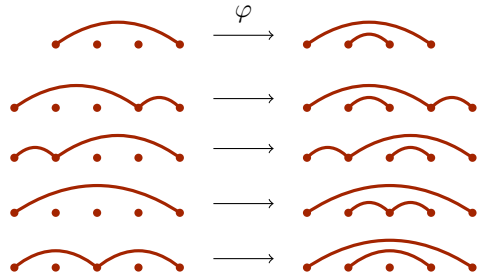
\def\ff{0.55}
\begin{gather*}
  \drawarc{\ff}{4}{1/4}\quad \tikz[baseline=-5]{\draw[->] (0,0)--(0.8,0); \node[above=1pt] at (0.4,0) {\small $\varphi$}}
  \quad \drawarc{\ff}{4}{1/4,2/3} \\[5pt]
  \drawarc{\ff}{5}{1/4,4/5}\quad \tikz[baseline=-5]{\draw[->] (0,0)--(0.8,0);}\quad \drawarc{\ff}{5}{1/4,4/5, 2/3} \\
  \drawarc{\ff}{5}{1/2,2/5}\quad \tikz[baseline=-5]{\draw[->] (0,0)--(0.8,0);}\quad \drawarc{\ff}{5}{1/2,2/5,3/4} \\
  \drawarc{\ff}{5}{1/5}\quad \tikz[baseline=-5]{\draw[->] (0,0)--(0.8,0);}\quad \drawarc{\ff}{5}{1/5,2/3,3/4} \\
  \drawarc{\ff}{5}{1/3,3/5}\quad \tikz[baseline=-5]{\draw[->] (0,0)--(0.8,0);}\quad \drawarc{\ff}{5}{1/5,2/4}
\end{gather*}

\begin{corollary}
Let $\Q_n$ be the set of partitions of $[n]$ whose indecomposable components are totally nested. This set is equinumerous with $\P_{\rm ncn}(n)$, so $|\Q_n|=\NCN(n)=F_{2n-1}$. Moreover, the elements of $\Q_n$ may be represented as tilings of an $n$-board corresponding to compositions of $n$ where part $k$ is decorated by an element of $\P_{\rm tn}(k)$. For example, $\Q_3$ consists of the 5 tilings

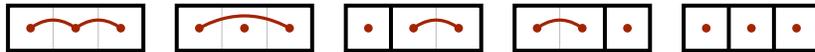
\[
\arcTile{3}{}{}{1/2,2/3}{0.6}\quad
\arcTile{3}{}{}{1/3}{0.6}\quad
\arcTile{3}{1}{}{2/3}{0.6}\quad
\arcTile{3}{2}{}{1/2}{0.6}\quad
\arcTile{3}{1,2}{}{}{0.6}
\]
\end{corollary}

\begin{theorem}\label{thm:oddFiboUnimodal}
The set $\C_{F_{2k-1}}(n)$ of compositions of $n$, where part $k$ comes in $F_{2k-1}$ colors, is in one-to-one correspondence with the set of unimodal sequences of length $n$ covering an initial interval of positive integers. These sets are enumerated by the sequence 1, 3, 10, 34, 116, 396, 1352, 4616, \dots, see \oeis{A007052}.
\end{theorem}
\begin{proof}
Since $F_{2k-1}=|\Q_k|$, the elements of $\C_{F_{2k-1}}(n)$ may be represented as tilings of an $n$-board whose building blocks are tiles, decorated with totally nested partitions, that are connected by either a dotted or a solid vertical line. For example,
\[ n=1:\;\; \parbox{4.5ex}{\arcTile{1}{}{}{}{0.6}}, \qquad n=2:\;\; \parbox{28ex}{$\arcTile{2}{}{}{1/2}{0.6}\quad \arcTile{2}{}{1}{}{0.6}\quad \arcTile{2}{1}{}{}{0.6}$}, \]
and for $n=3$, we have 10 tilings:
\smallskip
\begin{equation} 
\label{eq:oddFiboColor}
\begin{gathered}
\arcTile{3}{}{}{1/2,2/3}{0.6}\quad
\arcTile{3}{}{}{1/3}{0.6}\quad
\arcTile{3}{}{1}{2/3}{0.6}\quad
\arcTile{3}{}{2}{1/2}{0.6}\quad
\arcTile{3}{}{1,2}{}{0.6} \\
\arcTile{3}{1}{}{2/3}{0.6}\quad
\arcTile{3}{1}{2}{}{0.6}\quad
\arcTile{3}{2}{}{1/2}{0.6}\quad
\arcTile{3}{2}{1}{}{0.6}\quad
\arcTile{3}{1,2}{}{}{0.6}
\end{gathered}
\end{equation}

On the other hand, if $\pi=\{P_1,\dots,P_k\}\in \P_{\rm tn}(n)$, there is a unique sequence $(a_1,a_2,\dots,a_n)$ defined by 
\[ a_i = j \;\text{ if }\; i\in P_j. \] 
Every such sequence is unimodal and satisfies:
\begin{itemize}
\item $a_1=a_n=1$,
\item $|a_{i+1}-a_i|\le 1$ for every $i\in\{1,\dots,n-1\}$.
\end{itemize}
This gives a natural bijective map $\psi:  \P_{\rm tn}(n)\to \U_{\rm tn}(n)$, where $\U_{\rm tn}(n)$ denotes the set of unimodal sequences of length $n$ satisfying the above two properties. For example, the totally nested partition

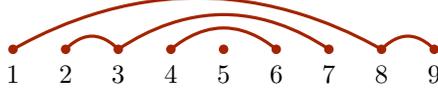
\begin{center}
\drawarc{0.7}{9}{1/8,8/9,2/3,3/7,4/6} \\
\tikz{
\foreach \i in {1,...,9}{
 \node at (0.7*\i,0) {\small \i};
}}
\end{center}
corresponds to the unimodal sequence $(1,2,2,3,4,3,2,1,1)$. 

For $u\in \U_{\rm tn}(k)$, let $i^*\le j^*$ be the indices of the first and last elements of $u$ such that $u_{i^*}=u_{j^*}=\max(u)$. For $v\in \U_{\rm tn}(\ell)$, we define two connecting operations:
\begin{enumerate}[(a)]
\item $u \oplus_l v=$ sequence of length $k+\ell$ obtained by inserting $v+\max(u)$ to the left of $u_{i^*}$,
\item $u \oplus_r v=$ sequence of length $k+\ell$ obtained by inserting $v+\max(u)$ to the right of $u_{j^*}$,
\end{enumerate}
where $v+\max(u)=(v_1+\max(u),\dots,v_\ell+\max(u))$. For example,
\begin{gather*}
 (1,1) \oplus_l (1,2,1) = (2,3,2,1,1), \qquad (1,1) \oplus_r (1,2,1) = (1,1,2,3,2), \\
 (1,2,1) \oplus_l (1,1) = (1,3,3,2,1), \qquad (1,2,1) \oplus_r (1,1) = (1,2,3,3,1).
\end{gather*}

We let $\U(n)$ be the set $\U_{\rm tn}(n)$ together with all the sequences of length $n$ that can be generated with the connecting associative operations $\oplus_l$ and $\oplus_r$. By definition, the elements of $\U(n)$ are unimodal sequences covering an initial interval of positive integers. Moreover, it is easy to verify that $\U(n)$ contains all such sequences. For example, $\U(3)$ consists of the 10 sequences:
\begin{center}
\def\decom#1{\scriptsize\color{blue}$#1$}
\begin{tabular}{ccccc}
(1,1,1) & (1,2,1) & (2,2,1) & (2,1,1) & (3,2,1) \\
\decom{(1,1,1)} & \decom{(1,2,1)} & \decom{(1)\oplus_l (1,1)} & \decom{(1,1)\oplus_l (1)} & \decom{(1)\oplus_l (1)\oplus_l (1)} \\[5pt]
(1,2,2) & (1,3,2) & (1,1,2) & (2,3,1) & (1,2,3) \\
\decom{(1)\oplus_r (1,1)} & \decom{(1)\oplus_r (1)\oplus_l (1)} & \decom{(1,1)\oplus_r (1)} & \decom{(1)\oplus_l (1)\oplus_r (1)} & \decom{(1)\oplus_r (1)\oplus_r (1)}
\end{tabular}
\end{center}
which are placed in order to match their corresponding compositions in \eqref{eq:oddFiboColor}.

Finally, the correspondences
\begin{gather*}
\tikz{\node[left=1pt] at (0,0){$\pi$}; \draw[dottedSep] (0,-0.2) -- (0,0.25); \node[right=1pt] at (0,0){$\sigma$};
 \draw[<->] (0.6,0)--(1.2,0); \node at (2.6,0) {$\psi(\pi) \oplus_l \psi(\sigma)$,}} \quad
\tikz{\node[left=1pt] at (0,0){$\pi$}; \draw[very thick] (0,-0.2) -- (0,0.25); \node[right=1pt] at (0,0){$\sigma$};
 \draw[<->] (0.6,0)--(1.2,0); \node at (2.6,0) {$\psi(\pi) \oplus_r \psi(\sigma)$,}}
\end{gather*}
for totally nested partitions $\pi$ and $\sigma$, give a bijection between $\C_{F_{2k-1}}(n)$ and $\U(n)$.
\end{proof}

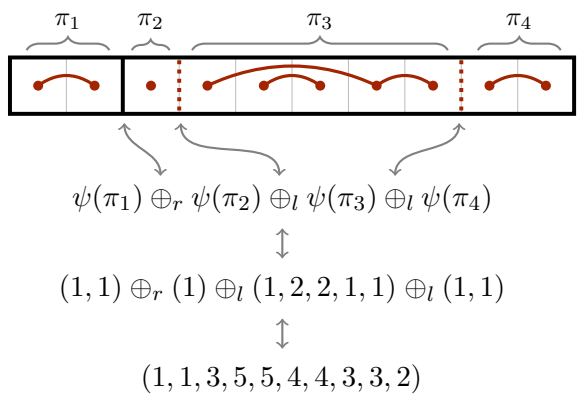
\begin{figure}[ht]
\begin{tikzpicture}
\draw [decorate,decoration={brace,amplitude=5pt,mirror,raise=2pt},thick,gray]
 (-2.3,0.4) -- (-3.4,0.4) node [midway,yshift=13pt,black] {\small $\pi_1$};
\draw [decorate,decoration={brace,amplitude=4pt,mirror,raise=2pt},thick,gray]
 (-1.5,0.4) -- (-2,0.4) node [midway,yshift=13pt,black] {\small $\pi_2$};
\draw [decorate,decoration={brace,amplitude=5pt,mirror,raise=2pt},thick,gray]
 (2.2,0.4) -- (-1.2,0.4) node [midway,yshift=13pt,black] {\small $\pi_3$};
\draw [decorate,decoration={brace,amplitude=5pt,mirror,raise=2pt},thick,gray]
 (3.7,0.4) -- (2.6,0.4) node [midway,yshift=13pt,black] {\small $\pi_4$};
\node at (0,0) {\arcTile{10}{2}{3,8}{1/2,4/7,5/6,7/8,9/10}{0.75}};
\draw[<->,gray,thick] (-2.1,-0.5) to[out=-90,in=90] (-1.62,-1.2);
\draw[<->,gray,thick] (-1.35,-0.5) to[out=-90,in=90] (-0.04,-1.2);
\draw[<->,gray,thick] (2.35,-0.5) to[out=270,in=90] (1.5,-1.2);
\node at (0,-1.5) {$\psi(\pi_1)\oplus_r \psi(\pi_2) \oplus_l \psi(\pi_3) \oplus_l \psi(\pi_4)$};
\draw[thick,<->,gray] (0,-1.9) -- (0,-2.3);
\node at (0,-2.7) {$(1,1)\oplus_r (1) \oplus_l (1,2,2,1,1) \oplus_l (1,1)$};
\draw[thick,<->,gray] (0,-3.1) -- (0,-3.5);
\node at (0,-3.9) {$(1,1,3,5,5,4,4,3,3,2)$};
\end{tikzpicture}
\caption{Example for the bijection given in Theorem~\ref{thm:oddFiboUnimodal}.}
\label{fig:oddFiboBijection}
\end{figure}

\bigskip
We finish this section with a connection between $\C_{F_{2k-1}}(n)$ and the set of order-consecutive partition sequences studied by Hwang and Mallows \cite{HM95}.

An ordered partition sequence $(S_1,\dots,S_p)$ of $[n]$ is said to be {\em order-consecutive} if and only if for every $k=1,\dots, p$, we have that $\bigcup_{i=1}^k S_i$ is a consecutive subset of $[n]$. For example,
\begin{equation} \label{eq:sampleOCPS}
 S_1=\{5,6\},\; S_2=\{7\},\; S_3=\{3,4\},\; S_4=\{2,8,9\},\; S_5=\{1\} 
\end{equation}
is an order-consecutive partition sequence of $\{1,\dots,9\}$ with 5 parts.

Hwang and Mallows \cite{HM95} cleverly represented order-consecutive partition sequences of $[n]$ with $p$ parts as strings created by inserting $p-1$ commas and $p-1$ slashes (in alternating fashion, from left to right, starting with a comma) into the $n$ spaces between the elements of $[n]$ (including the space after the last element). These strings describe how the partition is to be constructed: the commas divide $[n]$ into $p$ blocks that form a consecutive partition sequence of $[n]$, and the slashes indicate how the number of elements are to be placed at each step. In this representation $[n]$ is split by the commas such that the $i$th block has the same cardinality of $S_i$ for $i = 1,\dots, p$, and all the commas and slashes are placed, until we have $p-1$ commas and $p-1$ slashes, following these rules for $j = 1,\dots, p-1$:
\begin{itemize}
\item between the $j$th comma and the $j$th slash we have a number of elements corresponding to the number of elements of $S_{j+1}$ that lie to the left of those of $\bigcup_{i=1}^{j} S_i$;
\item between the $j$th slash and the $(j+1)$th comma we have a number of elements corresponding to the number of elements of $S_{j+1}$ that lie to the right of those of $\bigcup_{i=1}^{j} S_i$.
\end{itemize}
For example, the sequence in \eqref{eq:sampleOCPS} can be represented as
\[ 12,/3,45/,6/78,9/ \]

Let us point out that the subset of such representations that do not end with a slash, and for which commas and slashes do not appear next to each other, is in one-to-one correspondence with the set of totally nested partitions. Indeed, given a partition $\pi=\{B_1,\dots,B_p\}\in\P_{\rm tn}(n)$, the sequence of blocks listed in reversed order
\[ B_p,B_{p-1},\dots,B_1, \]
forms an order-consecutive partition sequence of $[n]$, and for every $j=1,\dots,p-1$, the block $B_j$ can be split into two subsets $B'_j\cup B''_j$ such that $B'_j<B_{j+1}<B''_j$. Therefore, the corresponding comma-slash representation of this sequence has neither a slash at the end nor a slash next to a comma. We let $\xi(\pi)$ be the comma-slash representation of the sequence $(B_p,B_{p-1},\dots,B_1)$.

For example, given the totally nested partition 
\begin{equation*}
\pi = \{\{1,8,9\},\{2,3,7\},\{4,6\},\{5\}\}, 
\end{equation*}
we get the order-consecutive partition sequence
\begin{equation*}
 \{5\}, \{4,6\}, \{2,3,7\}, \{1,8,9\},
\end{equation*}
which implies $\xi(\pi)=1,2/3,45/6,7/89$. The map $\xi$ is clearly invertible.

\begin{theorem}
The set $\C_{F_{2k-1}}(n)$ of compositions of $n$, where part $k$ comes in $F_{2k-1}$ colors, is in one-to-one correspondence with the set of order-consecutive partition sequences of $[n]$.
\end{theorem}
\begin{proof}
We give a bijection that relies on the map $\xi$ described above:
\begin{enumerate}[(i)]
\item Given an element of $\C_{F_{2k-1}}(n)$, represent it as a tiling of an $n$-board whose tiles are connected by dotted or solid vertical lines and are decorated with totally nested partitions. 
\item List the totally nested components of the tiling, placing a black comma-slash pair ``{\bf ,/}'' when tiles are joined by a solid line, and a red comma-slash pair ``\textcolor{tileColor}{\bf ,/}'' when tiles are joined by a dotted line.
\item For each totally nested component, apply the map $\xi$ and relabel its elements so that the resulting partition corresponds to a consecutive partition of $[n]$.
\item Slide any red slash to the right so that it creates a black ``/,'' pair with the next comma to its right. If a red slash has no comma to its right, then slide it all the way to the end.
\end{enumerate}

This algorithm gives an order-consecutive partition sequence of $[n]$, and it is reversible. 

For example, if $(2,7)$ is an element of $\C_{F_{2k-1}}(9)$, where part 2 is colored by \parbox{6ex}{\arcTile{2}{}{}{1/2}{0.45}} and part 7 is colored by \parbox{22ex}{\arcTile{7}{}{1,6}{2/5,3/4,5/6}{0.5}}, then the above algorithm gives the following steps:
\medskip
\begin{align*}
\text{(i) } &\; \parbox{40ex}{\arcTile{9}{2}{3,8}{1/2,4/7,5/6,7/8}{0.67}} \\[5pt]
\text{(ii) } &\; \{1,2\} \boldsymbol{,/} \{1\} \textcolor{tileColor}{\boldsymbol{,/}} \{1,4,5\}, \{2,3\} \textcolor{tileColor}{\boldsymbol{,/}} \{1\}\\[3pt]
\text{(iii) } &\; 12, /3\textcolor{tileColor}{,/} 45, 6/78\textcolor{tileColor}{,/}9 \\[3pt]
\text{(iv) } &\; 12, /3, 45/, 6/78, 9/
\end{align*}
The resulting string is the representation of the sequence $\{5,6\},\{7\},\{3,4\},\{2,8,9\},\{1\}$.
\end{proof}

\section{Further remarks}
\label{sec:furtherRemarks}

There is a connection between the colored compositions discussed in Section~\ref{sec:FiboColoring} and certain restricted 2-compositions recently considered by Hopkins and Ouvry \cite{HO20}.

A $k$-composition of $n$ is a standard composition of $n$ where each part is assigned a color $1,\dots,k$, and the first part must have color 1. Among the results provided in \cite{HO20}, the authors give generating functions for the enumeration of $k$-compositions whose parts are restricted to only parts $1$ and $2$ (denoted by $C^{k}_{12}(n)$), only odd parts (denoted by $C^{k}_{\rm odd}(n)$), and only parts greater than $1$ (denoted by $C^{k}_{\hat{1}}(n)$). For $k=2$, the cardinality of these sets match those of our colored compositions in Theorems~\ref{thm:noAdjacent0s}, \ref{thm:Pell}, and \ref{thm:Jacobsthal}, respectively.

For instance, when restricted to parts $1$ and $2$, the 2-compositions have only four possible parts, $1_{1}, 1_{2}, 2_{1}, $ and $ 2_{2}$. There is a simple bijection between these compositions and the colored compositions in $\C_{F_{k+1}}(n)$. Representing colored composition as tiled boards (see Section~\ref{sec:FiboColoring}), we assign the color $1$ in the composition to a solid separator and the color $2$ to a dotted separator. Then, since a 2-composition must start with color $1$, we begin from left to right on the outside of the board assigning colored numbers to each separator based on the size of the tile (including the secondary tiles) and the type of separator. We exclude the right end of the board. 

For example, the 8 elements of $C^{2}_{12}(3)$ are obtained from $\C_{F_{k+1}}(3)$ as follows: 

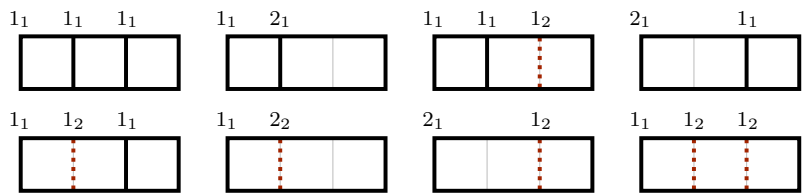
\begin{center}
\begin{tikzpicture}[scale=0.55]
\node at (-5,0.3) {\mixedTile{3}{1,2}{}{0.7}};
\node at (-6.7,1.35) {\scriptsize $1_{1}$};
\node at (-5.4,1.35) {\scriptsize $1_{1}$};
\node at (-4.1,1.35) {\scriptsize $1_{1}$};
\node at (0,0.3) {\mixedTile{3}{1}{}{0.7}};
\node at (-1.7,1.35) {\scriptsize $1_{1}$};
\node at (-0.4,1.35) {\scriptsize $2_{1}$};
\node at (5,0.3) {\mixedTile{3}{1}{2}{0.7}};
\node at (3.3,1.35) {\scriptsize $1_{1}$};
\node at (4.6,1.35) {\scriptsize $1_{1}$};
\node at (5.9,1.35) {\scriptsize $1_{2}$};
\node at (10,0.3) {\mixedTile{3}{2}{}{0.7}};
\node at (8.3,1.35) {\scriptsize $2_{1}$};
\node at (10.9,1.35) {\scriptsize $1_{1}$};
\end{tikzpicture}

\begin{tikzpicture}[scale=0.55]
\node at (-5,0.3) {\mixedTile{3}{2}{1}{0.7}};
\node at (-6.7,1.35) {\scriptsize $1_{1}$};
\node at (-5.4,1.35) {\scriptsize $1_{2}$};
\node at (-4.1,1.35) {\scriptsize $1_{1}$};
\node at (0,0.3) {\mixedTile{3}{}{1}{0.7}};
\node at (-1.7,1.35) {\scriptsize $1_{1}$};
\node at (-0.4,1.35) {\scriptsize $2_{2}$};
\node at (5,0.3) {\mixedTile{3}{}{2}{0.7}};
\node at (3.3,1.35) {\scriptsize $2_{1}$};
\node at (5.9,1.35) {\scriptsize $1_{2}$};
\node at (10,0.3) {\mixedTile{3}{}{1,2}{0.7}};
\node at (8.3,1.35) {\scriptsize $1_{1}$};
\node at (9.6,1.35) {\scriptsize $1_{2}$};
\node at (10.9,1.35) {\scriptsize $1_{2}$};
\end{tikzpicture}
\end{center}

Clearly, this map also gives bijections $\C_{F_{k}}(n)\to C^{2}_{\rm odd}(n)$ and  $\C_{F_{k-1}}(n)\to C^{2}_{\hat{1}}(n)$.

More generally, using the above map, the elements of $C^k(n)$ (set of $k$-compositions of $n$) could be represented as tilings of an $n$-board whose tiles are put together using $k$ different types of separators. It would be interesting to use this refined tiling representation to investigate $k$-compositions with other part restrictions and colorings.


\end{document}